\documentclass[12pt]{amsart}
\usepackage[none]{hyphenat}
\usepackage[utf8]{inputenc}
\usepackage[T1]{fontenc}
\usepackage{amsfonts}
\usepackage[leqno]{amsmath}
\usepackage{amssymb, comment, float}
\usepackage[dvipsnames]{xcolor}
\usepackage[foot]{amsaddr}
\usepackage[shortlabels]{enumitem}
\usepackage{mathdots}
\usepackage[footskip=0.5in,  headheight = 0.5in, top=1.25in, bottom=1.25in,  right=1in,  left=1in]{geometry}
\usepackage{hyperref}
\hypersetup{
colorlinks,
linkcolor=black,
urlcolor=Blue,
citecolor=black,}
\usepackage[center]{caption}
\usepackage{eufrak}
\usepackage{marvosym}     
\usepackage{latexsym}
\usepackage[nodayofweek]{datetime}
\usepackage{tikz}
\usepackage{subfig}
\usepackage{thm-restate}

\newtheorem{theorem}{Theorem}[section]
\newtheorem{lemma}[theorem]{Lemma}

\newtheorem{question}[theorem]{Question}

\DeclareMathOperator{\tw}{tw}

\DeclareMathOperator{\T}{{\text{\sf{T}}}}

\def\dd{\hbox{-}}   

\newcommand{\mf}{\mathfrak}
\newcommand{\mca}{\mathcal}
\newcommand{\poi}{\mathbb{N}} 

\newcommand{\pre}{\preccurlyeq}

\newcounter{tbox}
\newcommand{\sta}[1]{\medskip\medskip\refstepcounter{tbox}\noindent{\parbox{\textwidth}{(\thetbox) \emph{#1}}}\vspace*{0.3cm}}
\newcommand{\mylongtitle}[1]{%
  \ifodd\value{page}%
    \protect\parbox{0.97\linewidth}{#1}\hfill%
  \else%
    \hfill\protect\parbox{0.97\linewidth}{#1}%
  \fi%
}

\title[Induced subgraphs and tree decompositions XVII.]{Induced subgraphs and tree decompositions XVII. Anticomlete sets of large treewidth}

\author{Maria Chudnovsky$^{\dagger \ast}$}
\author{Sepehr Hajebi$^{\mathsection}$}
\author{Sophie Spirkl$^{\mathsection \parallel}$}

\thanks{$^{\dagger}$ Princeton University, Princeton, NJ, USA}
\thanks{$^{\mathsection}$ Department of Combinatorics and Optimization, University of Waterloo, Waterloo, Ontario, Canada}
\thanks{$^{\ast}$ Supported by  NSF-EPSRC Grant DMS-2120644, AFOSR grant FA9550-22-1-0083 and NSF Grant DMS-2348219.} 
\thanks{$^{\parallel}$ We acknowledge the support of the Natural Sciences and Engineering Research Council of Canada (NSERC), [funding reference number RGPIN-2020-03912].
Cette recherche a \'et\'e financ\'ee par le Conseil de recherches en sciences naturelles et en g\'enie du Canada (CRSNG), [num\'ero de r\'ef\'erence RGPIN-2020-03912]. This project was funded in part by the Government of Ontario. This research was conducted while Spirkl was an Alfred P. Sloan Fellow.}
\date {\today}
\begin{document}
\maketitle
\sloppy

\begin{abstract}
Two sets $X, Y$ of vertices in a graph $G$ are \textit{anticomplete} if $X\cap Y=\varnothing$ and there is no edge in $G$ with an end in $X$ and an end in $Y$. We prove that every graph $G$ of sufficiently large treewidth contains two anticomplete sets of vertices each inducing a subgraph of large treewidth unless $G$ contains, as an induced subgraph, a highly structured graph of large treewidth that is an obvious counterexample to this statement. These are: complete subgraphs, complete bipartite graphs and \textit{interrupted $s$-constellations}. The latter is a slightly adjusted version of a well-known construction by Bonamy et al.
\end{abstract}

\section{Introduction}\label{sec:intro}
The set of all positive integers is denoted by $\poi$. Graphs in this paper have finite vertex sets, no loops and no parallel edges. Given a graph $G=(V(G), E(G))$ and $X\subseteq V(G)$, we use both $X$ and $G[X]$ to denote the induced subgraph of $G$ with vertex set $X$ (also called the \textit{subgraph of $G$ induced by $X$}). For a graph $H$, we say that a graph $G$ is \textit{$H$-free} if $G$ has no induced subgraph isomorphic to $H$. For standard graph theoretic terminology, see \cite{diestel}.

This paper continues a study of the interplay between the induced subgraphs of a graph $G$ and its treewidth (denoted $\tw(G)$; see \cite{diestel} for a definition). Treewidth is a measure of structural complexity in graphs. Graphs of small treewidth are restricted in their \textit{global} structure, and this restriction wanes as the treewidth increases. Once the treewidth becomes sufficiently large, then there likely is something to say about the \textit{local} structure of the graph. For instance, the \textit{Grid Theorem} of Robertson and Seymour \cite{GMV}, Theorem~\ref{thm:wallminor} below, says that every graph of large enough treewidth has a subgraph of large treewidth that is planar.

\begin{theorem}[Robertson and Seymour \cite{GMV}]\label{thm:wallminor}
For every integer $r\in \poi$,
there is a constant $f_{\ref{thm:wallminor}}=f_{\ref{thm:wallminor}}(r)\in \poi$ such that every graph $G$ with $\tw(G) \geq f_{\ref{thm:wallminor}}$ has a subgraph isomorphic to a subdivision of $W_{r\times r}$.
\end{theorem}
(We write $W_{r \times r}$ to denote the $r$-by-$r$ hexagonal grid, also known as the $r$-by-$r$ \textit{wall}. It is well-known \cite{diestel} that every subdivision of $W_{r\times r}$ has treewidth $r$, and so does the line graph of every subdivision of $W_{r\times r}$; see Figure~\ref{fig:basic}.)
  \begin{figure}[t!]
        \centering
\includegraphics[width=0.8\textwidth]{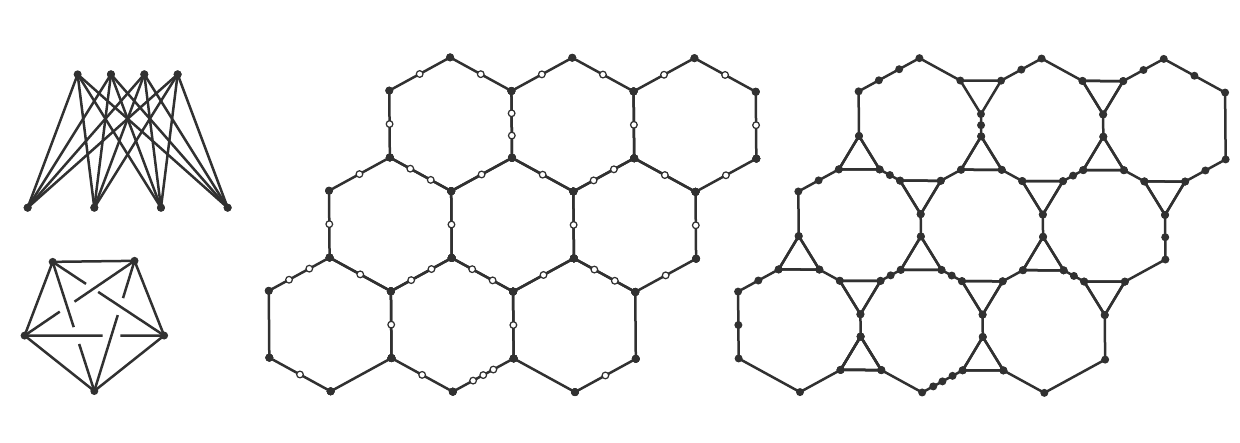}
        \caption{Examples of graphs with large treewidth: $K_{4,4}$ (top left), $K_5$ (bottom left), a subdivision of the $W_{4\times 4}$ and the line graph of a subdivision of the $W_{4\times 4}$ (right).}
        \label{fig:basic}
    \end{figure}   

A primary goal in this series of papers is to explore the analog of Theorem~\ref{thm:wallminor} for induced subgraphs. On that note, we would often like to answer questions of the following generic form: Must an arbitrary graph of sufficiently large treewidth contain an induced subgraph of relatively large treewidth that is ``special'' in some way? This paper in particular pursues a line of work on characterizing the induced subgraph obstructions to bounded treewidth under some density or sparsity assumption. For instance, a result of Korhonen \cite{korhonen2023grid} shows that, apart from subdivided walls and their line graphs, all other obstructions contain at least one vertex of large degree. An easy consequence of a result of K\"uhn and Osthus \cite{kuhn2004induced} is that, apart from complete and complete bipartite graphs, all obstructions have bounded degeneracy.

Here, in analogy to a conjecture of El-Zahar and Erd\H{o}s \cite{eez} about graphs with large chromatic number, we ask: when does a graph of large enough treewidth have an induced subgraph with at least two components of large treewidth? Our main result, Theorem \ref{thm:EEZ_isg}, answers this question.

For a graph $G$, we say that $X, Y\subseteq V(G)$ are \textit{anticomplete in $G$} if $X\cap Y=\varnothing$ and there is no edge in $G$ with an end in $X$ and an end in $Y$ (if $X=\{x\}$ is a singleton, then we also say \textit{$x$ is anticomplete to $Y$ in $G$}).

\begin{question}\label{q:main}
    When does a graph of large enough treewidth contain two anticomplete sets of vertices each inducing a subgraph of large treewidth?
\end{question}

Not always. It is well-known \cite{diestel} that for every $t\in \poi$, both the complete graph $K_{t+1}$ and the complete bipartite graph $K_{t,t}$ have treewidth $t$ (see Figure~\ref{fig:basic}). However, for every two anticomplete sets $X,Y$ in a complete or a complete bipartite graph, either $X$ or $Y$ is a stable set (where a \textit{stable set} in a graph $G$ is a subset of $V(G)$ that contains no two vertices adjacent in $G$).
\medskip

Our first result proves that complete graphs and complete bipartite graphs are the only culprits (recall that an \textit{induced minor} of a graph is obtained by only removing vertices and contracting edges, unlike \textit{minors} which also allows for removing edges):

\begin{theorem}\label{thm:EEZ_ind_minor}
    For all $a,b,c\in \poi$, there is a constant $f_{\ref{thm:EEZ_ind_minor}}=f_{\ref{thm:EEZ_ind_minor}}(a,b,c)\in \poi$ such that for every graph $G$ with $\tw(G)>f_{\ref{thm:EEZ_ind_minor}}$, one of the following holds.
    \begin{enumerate}[{\rm (a)}]
        \item\label{thm:EEZ_ind_minor_a} There is an (induced) subgraph of $G$ isomorphic to $K_{a}$.
        \item\label{thm:EEZ_ind_minor_b} There is an induced minor of $G$ isomorphic to $K_{b,b}$.
        \item\label{thm:EEZ_ind_minor_c} There are anticomplete subsets $X,Y$ of $V(G)$ with $\tw(X),\tw(Y)\geq c$.
        \end{enumerate}
\end{theorem}

As mentioned earlier, Theorem~\ref{thm:EEZ_ind_minor} is reminiscent of a conjecture by El-Zahar and Erd\H{o}s \cite{eez} that every graph $G$ of sufficiently large chromatic number contains two anticomplete sets each inducing a subgraph of large chromatic number, unless there is a large complete subgraph in $G$. This would be sharp because complete graphs have arbitrarily large chromatic number (and no two non-empty anticomplete sets of vertices). 

Likewise, Theorem~\ref{thm:EEZ_ind_minor} is sharp in the sense that both complete graphs and complete bipartite graphs are unavoidable outcomes. Nevertheless, one may ask whether it is necessary for the complete bipartite outcome in \ref{thm:EEZ_ind_minor}\ref{thm:EEZ_ind_minor_b} to appear as an ``induced minor'' (and not an ``induced subgraph''). The answer is ``yes,'' provided by an explicit construction of $(K_4,K_{3,3})$-free graphs with arbitrarily large treewidth and no two anticomplete induced subgraphs of treewidth $3$ or more.  We call these graphs ``interrupted $s$-constellations,'' and their exact definition is given below.
\medskip

For every integer $k$, we denote by $\poi_k$ the set of all positive integers at most $k$ (so $\poi_k=\varnothing$ if $k\leq 0$). Let $P$ be a graph which is a path. Then we write $P=p_1\dd \cdots\dd p_k$ to mean $V(P)=\{p_1,\ldots,p_k\}$ for $k\in \poi$, and $E(P)=\{p_ip_{i+1}:i\in \poi_{k-1}\}$. We call $p_1,p_k$ the \textit{ends} of $P$, and we call $P\setminus \{p_1,p_k\}$ the interior of $P$, denoted $P^*$. For vertices $u,v\in V(P)$, we denote by $u\dd P\dd v$ the subpath of $P$ from $u$ to $v$. Recall that the \textit{length} of a path is its number of edges. It follows that a path $P$ has distinct ends if and only if $P$ has non-zero length, and $P$ has non-empty interior if and only if $P$ has length at least two. Given a graph $G$, by a {\em path in $G$} we mean an induced subgraph of $G$ which is a path.

Let $s\in \poi$. An \textit{$s$-constellation} is a graph $\mf{c}$ for which there is a path $L$ in $\mf{c}$ such that $S=\mf{c}\setminus L$ is a stable set of cardinality $s$ in $G$ and every vertex in $S$ has at least one neighbor in $L$. We denote $\mf{c}$ by the pair $(S,L)$. We say that $\mf{c}$ is \textit{ample} if no two vertices in $S$ have a common neighbor in $L$. We say that $\mf{c}$ is \textit{interrupted} if the vertices in $S$ can be enumerated as $x_1,\ldots, x_s$ such that for all $i,j,k\in \poi_s$ with $i<j<k$ and every path $R$ in $\mf{c}$ from $x_i$ to $x_j$ with $R^*\subseteq L$, the vertex $x_k$ has a neighbor in $R^*$. See Figure~\ref{fig:interrupted}.
\begin{figure}[t!]
    \centering
    \includegraphics[width=\linewidth]{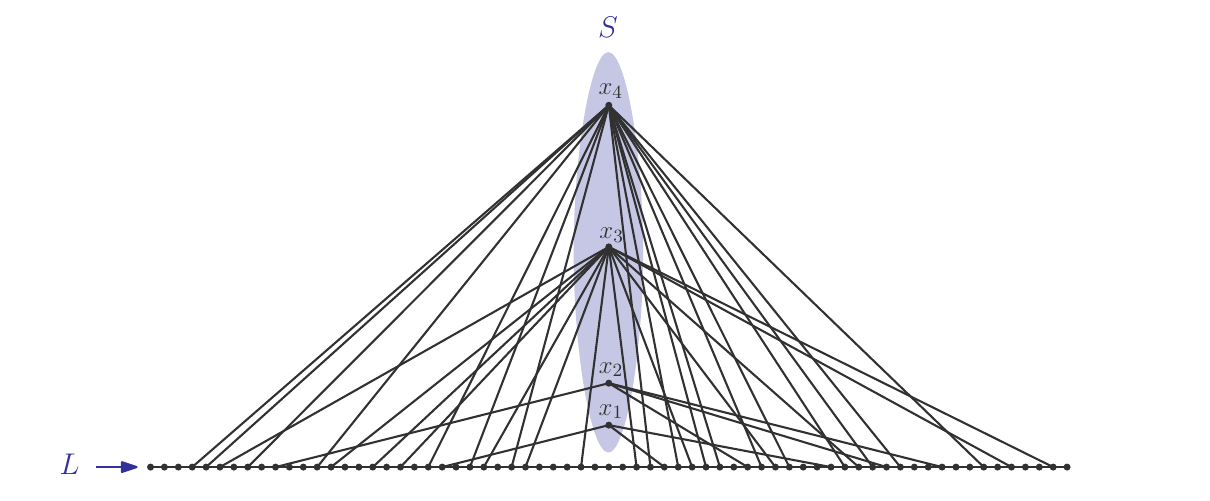}
    \caption{An ample interrupted $4$-constellation.}
    \label{fig:interrupted}
\end{figure}
\medskip

Interrupted $s$-constellations are a slightly adjusted version of a previous construction by Bonamy et al. \cite{deathstar} (see also \cite{tw9}). It is easily observed that for all $s\in \poi$, every ample interrupted $(2s+1)$-constellation is a $K_4$-free $K_{3,3}$-free graph with an induced minor isomorphic to $K_{s,s}$, and so with treewidth at least $s$. Moreover, it is proved in Lemma~9.2 from \cite{tw16} (and the proof is quite straightforward) that if $\mf{c}=(S,L)$ is an ample interrupted $s$-constellation for some $s\in \poi$, then for every two anticomplete induced subgraphs $X_1,X_2$ of $\mf{c}$, there exists $i\in \{1,2\}$ such that each component of $X_i$ intersects $S$ in at most one vertex. In particular, $\mf{c}$ has no two anticomplete induced subgraphs each of treewidth more than $2$.

It turns out that ample interrupted $s$-constellations are the only missing piece from a full answer to Question~\ref{q:main} with only induced subgraph outcomes. Our main result in this paper says that:

\begin{restatable}{theorem}{eezisg}\label{thm:EEZ_isg}
    For all $c,s,t\in \poi$, there is a constant $f_{\ref{thm:EEZ_isg}}=f_{\ref{thm:EEZ_isg}}(c,s,t)\in \poi$ such that for every graph $G$ with $\tw(G)>f_{\ref{thm:EEZ_isg}}$, one of the following holds.
    \begin{enumerate}[{\rm (a)}]
        \item\label{thm:EEZ_isg_a} There is an induced subgraph of $G$ isomorphic to $K_{t+1}$ or $K_{t,t}$.
        \item\label{thm:EEZ_isg_b} There is an induced subgraph of $G$ which is an ample interrupted $s$-constellation.
         \item \label{thm:EEZ_isg_c} There are anticomplete subsets $X,Y$ of $V(G)$ with $\tw(X),\tw(Y)\geq c$.
        \end{enumerate}
\end{restatable}

Most of the rest of this paper is devoted to the proof of Theorem~\ref{thm:EEZ_ind_minor}. In the last section, we will derive Theorem~\ref{thm:EEZ_isg} from \ref{thm:EEZ_ind_minor}; this is still far from trivial and relies in particular on the main result of our previous paper in this series \cite{tw16}. In contrast, it is easy to deduce \ref{thm:EEZ_ind_minor} from \ref{thm:EEZ_isg} because, once again, for every $b\in \poi$, every ample interrupted $(2b+1)$-constellation has an induced minor isomorphic to $K_{b,b}$. We also remark that Theorem~\ref{thm:EEZ_isg} provides an alternative proof for one of the two steps comprising the proof of the main result of \cite{tw9}.

\section{Models, strong blocks and an outline of the proof}\label{sec:model}

In this section, we first set up our terminology of ``models'' and ``strong blocks,'' and then give an outline of the proof of Theorem~\ref{thm:half_ind_main}.

Let $G$ be a graph and let $\mu\in \poi$. A \textit{$\mu$-model in $G$} is a $\mu$-tuple $K=(C_1,\ldots, C_{\mu})$ of pairwise disjoint connected induced subgraphs of $G$ such that for all distinct $i,j \in \poi_{\mu}$, the sets $C_i$ and $C_j$ are not anticomplete in $G$. We call $C_1,\ldots, C_{\mu}$ the \textit{branch sets} of $K$, and write 
$$V(K)=\bigcup_{i=1}^{\mu}C_i.$$
We say that $K$ is \emph{linear} if every branch set of $K$ is a path in $G$. 
\medskip

Let $G$ be a graph and let $\rho, \sigma\in \poi$. By a \emph{$(\rho, \sigma)$-model in $G$} we mean a $(\rho+\sigma)$-tuple $M=(A_1,\ldots, A_{\rho}; B_1,\ldots, B_{\sigma})$ of pairwise disjoint connected induced subgraphs of $G$ such that for all $i \in \poi_{\rho}$ and $j \in \poi_{\sigma}$, the sets $A_i$ and $B_j$ are not anticomplete in $G$. We call $A_1, \dots, A_{\rho}, B_1, \dots, B_{\sigma}$ the \emph{branch sets} of $M$.

 Let $M=(A_1,\ldots, A_{\rho}; B_1,\ldots, B_{\sigma})$ be a $(\rho, \sigma)$-model in $G$. We denote by $M^{\T}$ the $(\sigma,\rho)$-model $(B_1,\ldots, B_{\sigma}; A_1,\ldots, A_{\rho})$ in $G$. We say that $M$ is \emph{$A$-induced} if $A_1, \dots, A_{\rho}$ are pairwise anticomplete in $G$, and that $M$ is  \emph{$B$-induced} if $M^{\T}$ is $A$-induced. We say that $M$ is \textit{induced} if $M$ is both $A$-induced and $B$-induced. Note that $G$ has a minor isomorphic to $K_{\rho,\sigma}$ if and only if there is a $(\rho, \sigma)$-model in $G$, and $G$ has an induced minor isomorphic to $K_{\rho,\sigma}$ if and only if there is an induced $(\rho, \sigma)$-model in $G$.  

We say that $M$ is \emph{$A$-linear} if $A_i$ is a path in $G$ for every $i\in \poi_{\rho}$, and we say that $M$ is $B$ is  \emph{$B$-linear} if $M^{\T}$ is $A$-linear. We say that $M$ is \emph{linear} if it is both $A$-linear and $B$-linear. We also define $$A(M)=\bigcup_{i=1}^{\rho}A_i$$ and $$B(M)=\bigcup_{j=1}^{\sigma}B_j.$$
\medskip

Let $a,b,c\in \poi$. An \textit{$(a,b,c)$-candidate} is a graph $G$ such that:
 \begin{itemize}
        \item $G$ is $K_{a}$-free;
        \item there is no induced $(b,b)$-model in $G$;
         \item there are no anticomplete subsets $A,B$ of $V(G)$ with $\tw(A),\tw(B)\geq c$.
        \end{itemize}

Then Theorem~\ref{thm:EEZ_ind_minor} is equivalent to the Theorem~\ref{thm:EEZ_candidate} below. We will prove Theorem~\ref{thm:EEZ_candidate} in Section~\ref{sec:halftofull}.

\begin{restatable}{theorem}{candidate}\label{thm:EEZ_candidate}
    For all $a,b,c\in \poi$, there is a constant $f_{\ref{thm:EEZ_candidate}}=f_{\ref{thm:EEZ_candidate}}(a,b,c)\in \poi$ such that every $(a,b,c)$-candidate has treewidth at most $f_{\ref{thm:EEZ_candidate}}$.
\end{restatable}

Let $G$ be a graph. For a set $\mca{X}$ of subsets of $V(G)$, we write $V(\mca{X})=\bigcup_{X\in \mca{X}}X$. For $k,l\in \poi$, a \textit{$(k,l)$-block} in $G$ is a pair $(B, \mca{P})$ where $B\subseteq V(G)$ with $|B|\geq k$ and $\mca{P}:{B\choose 2}\rightarrow 2^{V(G)}$ is map such that $\mca{P}_{\{x,y\}}=\mca{P}(\{x,y\})$, for each $2$-subset $\{x,y\}$ of $B$, is a set of at least $l$ pairwise internally disjoint paths in $G$ from $x$ to $y$. We say that $(B,\mca{P})$ is \textit{strong} if for all distinct $2$-subsets $\{x,y\}, \{x',y'\}$ of $B$, we have $V(\mca{P}_{\{x,y\}})\cap V(\mca{P}_{\{x',y'\}})=\{x,y\}\cap\{x',y'\}$; that is, each path $P\in \mca{P}_{\{x,y\}}$ is disjoint from each path $P'\in \mca{P}_{\{x',y'\}}$, except when $P$ and $P'$ may share an end.

Let $t\in \poi$. We say that a graph $G$ is \textit{$t$-clean} if $G$ has no induced subgraph isomorphic to any of the following: $K_{t+1}$, $K_{t,t}$, a subdivision of $W_{t\times t}$, or the line graph of a subdivision of $W_{t\times t}$. The following was proved in \cite{tw7}.

\begin{theorem}[Abrishami, Alecu, Chudnovsky, Hajebi, Spirkl \cite{tw7}]\label{thm:noblock}
For all $k,l,t\in \poi$, there is a constant $f_{\ref{thm:noblock}}=f_{\ref{thm:noblock}}(k,l,t)\in \poi$ such that for every $t$-clean graph $G$ with $\tw(G)>f_{\ref{thm:noblock}}$,  there is a strong $(k,l)$-block in $G$.
\end{theorem}

We conclude this section with a sketch of the proof of Theorem~\ref{thm:EEZ_candidate}, as follows. Using Theorem~\ref{thm:noblock}, we deduce that our graph contains a strong $(k, l)$-block for large $k,l$. If this block is ``induced,'' then we have an induced subdivision of a large complete graph, and we are done. There are four kinds of edges that might make a strong block non-induced: In Section \ref{sec:comp}, we use Ramsey-type methods to show that we can restrict our attention to only those edges that are between paths; this implies the presence of a linear $\mu$-model $K^0$ (for some large $\mu$ depending on $k$ and $l$). 

In Section \ref{sec:half}, our goal is to extract from $K^0$ a large half-induced complete bipartite induced minor model. We proceed in two steps: first, we arrange that each vertex in $V(K^0)$ has neighbors in only a few other branch sets (or we find a model in which branch sets on one side are singletons, and so we can make it half-induced using Ramsey's theorem). Second, we use ideas similar to ``alignments'' from previous papers in this series \cite{tw9,tw11,tw12}: For a given branch set $P$, if the other sets attach to it in a simple way (which we call ``aligned''), we use this to control how $P$ attaches to other branch sets. Otherwise, we split $P$ into multiple paths, using these as the ``$B$-side'' of the half-induced model we were looking for in the first place. 

In Section \ref{sec:halftofull} we again use the idea of alignments. Roughly speaking, if the paths on the $A$-side have few edges between them, we get two anticomplete induced subgraphs of large treewidth; if there are many edges, we split two paths into multiple segments, obtaining an induced complete bipartite minor model consisting of these branch sets.  This will complete the proof of Theorem~\ref{thm:EEZ_candidate}.

\section{Complete models} \label{sec:comp}
Our main result in this section is the following:

\begin{restatable}{theorem}{compmodel}\label{thm:comp_model}
     For all $a,b,c,\mu\in \poi$, there is a constant $f_{\ref{thm:comp_model}}=f_{\ref{thm:comp_model}}(a,b,c,\mu)\in \poi$ such for every $(a,b,c)$-candidate $G$ with $\tw(G)>f_{\ref{thm:comp_model}}$, there is a linear $\mu$-model in $G$.
\end{restatable}

We will deduce Theorem~\ref{thm:comp_model} from the following more general result. (Recall that a \textit{proper} subdivision is one in which every edge is subdivided at least once.)

\begin{theorem}\label{thm:comp_model_general}
     For all $s,t,\mu, \rho,\sigma\in \poi$, there is a constant $f_{\ref{thm:comp_model_general}}=f_{\ref{thm:comp_model_general}}(s,t,\mu, \rho,\sigma)\in \poi$ such that for every $t$-clean graph $G$ with $\tw(G)>f_{\ref{thm:comp_model_general}}$, one of the following holds.
\begin{enumerate}[{\rm (a)}]
 \item\label{thm:comp_model_general_a}  There is an induced subgraph of $G$ isomorphic to a proper subdivision of $K_s$.
        \item \label{thm:comp_model_general_b} There is a linear induced $(\rho,\sigma)$-model in $G$. 
          \item\label{thm:comp_model_general_c}  There is a linear $\mu$-model in $G$.
\end{enumerate}
\end{theorem}

The proof relies on Theorem~\ref{thm:noblock} and two versions of Ramsey's theorem:

\begin{theorem}[Ramsey \cite{multiramsey}]\label{thm:multiramsey}
For all $l,m,n\in \poi$, there is a constant $f_{\ref{thm:multiramsey}}=f_{\ref{thm:multiramsey}}(l,m,n)\in \poi$ with the following property. Let $U$ be a set of cardinality at least $f_{\ref{thm:multiramsey}}$ and let $F$ be a non-empty set of cardinality at most $l$. Let $\Phi:\binom{U}{m}\rightarrow F$ be a map. Then there exist $i\in F$ and an $n$-subset $Z$ of $U$ such that $\Phi(X)=i$ for all $X\in \binom{Z}{m}$.
\end{theorem}

\begin{theorem}[Graham, Rothschild and Spencer \cite{productramsey}]\label{thm:productramsey}
For all  $n,q,r\in \poi$, there is a constant  $f_{\ref{thm:productramsey}}=f_{\ref{thm:productramsey}}(n,q,r)\in \poi$ with the following property. Let $U_1,\ldots, U_n$ be $n$ sets, each of cardinality at least $f_{\ref{thm:productramsey}}$ and let $W$ be a non-empty set of cardinality at most $r$. Let $\Phi$ be a map from the Cartesian product $U_1\times \cdots \times U_n$ to $W$. Then there exist $i\in W$ and a $q$-subset $Z_j$ of $U_j$ for each $j\in \poi_{n}$, such that for every $z\in Z_1\times \cdots\times Z_n$, we have $\Phi(z)=i$.
\end{theorem}

It is convenient to have an explicit bound in Theorem~\ref{thm:multiramsey} where $l=m=2$:

\begin{theorem}[Ramsey \cite{multiramsey}]\label{thm:classicalramsey}
For all $c,s\in \poi$, every graph on at least $c^s$ vertices has either a clique of cardinality $c$ or a stable set of cardinality $s$.
\end{theorem}

The following lemma is the key tool in our proof of Theorem~\ref{thm:comp_model_general}:

\begin{lemma}\label{lem:comp_model_rigid}
     For all $s,t,\rho,\sigma\in \poi$, there are constants $f_{\ref{lem:comp_model_rigid}}=f_{\ref{lem:comp_model_rigid}}(s,t,\rho,\sigma)\in \poi$ and $g_{\ref{lem:comp_model_rigid}}=g_{\ref{lem:comp_model_rigid}}(s,\rho,\sigma)\in \poi$ with the following property. Let $G$ be a $K_{t+1}$-free graph and let $(B,\mca{Q})$ be a strong $(f_{\ref{lem:comp_model_rigid}},g_{\ref{lem:comp_model_rigid}})$-block in $G$ such that  for every $\{x,y\}\subseteq B$, the paths $(Q^*: Q\in \mca{Q}_{\{x,y\}})$ are pairwise anticomplete in $G$. Then one of the following holds.
\begin{enumerate}[{\rm (a)}]
 \item\label{lem:comp_model_rigid_a}  There is an induced subgraph of $G$ isomorphic to a proper subdivision of $K_s$.
    \item \label{lem:comp_model_rigid_b} There is a linear induced $(\rho,\sigma)$-model in $G$. 
\end{enumerate}
\end{lemma}
\begin{proof}
Let $b=\max\{\rho,\sigma\}$ and let
$$\kappa=f_{\ref{thm:multiramsey}}(8,3,\max\{5b,s\}).$$

We will prove that
$$f_{\ref{lem:comp_model_rigid}}=f_{\ref{lem:comp_model_rigid}}(s,t,\rho,\sigma)=(t+1)^{\kappa}$$
and
$$g_{\ref{lem:comp_model_rigid}}=g_{\ref{lem:comp_model_rigid}}(s,\rho,\sigma)=f_{\ref{thm:productramsey}}\left(\binom{\kappa}{2},b, 2^{\displaystyle\binom{\binom{\kappa}{2}}{2}}\right)$$
satisfy the lemma.

Let $G$ be a $K_{t+1}$-free graph and let $(B,\mca{Q})$ be a strong $(f_{\ref{lem:comp_model_rigid}},g_{\ref{lem:comp_model_rigid}})$-block in $G$ such that  for every $\{x,y\}\subseteq B$, the paths $(Q^*: Q\in \mca{Q}_{\{x,y\}})$ are pairwise anticomplete in $G$. Suppose that \ref{lem:comp_model_rigid}\ref{lem:comp_model_rigid_b} does not hold; that is, there is no linear induced $(\rho,\sigma)$-model in $G$. Since $G$ is $K_{t+1}$-free, it follows from Theorem~\ref{thm:classicalramsey} and the choice of $f_{\ref{lem:comp_model_rigid}}$ that there is a stable set $S\subseteq B$ in $G$ with $|S|=\kappa$. In particular, for every $2$-subset $\{x,y\}$ of $S$, the paths in $\mca{Q}_{\{x,y\}}$ have nonempty interiors.

Now, we use the choice of $g_{\ref{lem:comp_model_rigid}}$ to show that:

\sta{\label{st:antiint2} For every $2$-subset $\{x,y\}$ of $S$, there exists $Q_{\{x,y\}}\in \mca{Q}_{\{x,y\}}$ such that for all distinct $2$-subsets $\{x,y\}, \{x',y'\}$ of $S$,  the paths $Q^*_{\{x,y\}}$ and $Q^*_{\{x',y'\}}$ are anticomplete in $G$.}

Let $\gamma=\binom{\kappa}{2}$. Then we have
$$g_{\ref{lem:comp_model_rigid}}=f_{\ref{thm:productramsey}}\left(\gamma, b, 2^{\binom{\gamma}{2}}\right).$$

Fix an enumeration $S_1,\ldots, S_{\gamma}$ of all $2$-subset of $S$, and write $\mca{Q}_i=\mca{Q}_{S_i}$ for every $i\in \poi_{\gamma}$. For each $z=(Q_1, \ldots, Q_{\gamma})\in \mca{Q}_1\times \cdots \times \mca{Q}_{\gamma}$, let $\Phi(z)$ be the set of all $2$-subsets $\{i,j\}$ of $\poi_{\gamma}$ such that the paths $Q^*_i$ and $Q^*_j$ are not anticomplete in $G$. Then
$$\Phi:\mca{Q}_1\times \cdots \times \mca{Q}_{\gamma}\rightarrow 2^{\displaystyle\binom{\poi_{\gamma}}{2}}$$
is a well-defined map. From Theorem~\ref{thm:productramsey} and the choice of $g_{\ref{lem:comp_model_rigid}}$, we deduce that there is a set $I$ of $2$-subsets of $\poi_{\gamma}$, as well as a $b$-subset $\mca{Z}_i=\{Z_{i,1},\ldots,Z_{i,b}\}$ of $\mca{Q}_i$ for each $i\in \poi_{\gamma}$, such that for every $z\in \mca{Z}_1\times \cdots\times \mca{Z}_{\gamma}$, we have $\Phi(z)=I$.

We further claim that $I=\varnothing$. For suppose that some $2$-subset $\{i,j\}$ of $\poi_{\gamma}$ belongs to $I$. Then, from the definition of $\Phi$ and the choice of $I$, it follows that for all $i',j'\in \poi_{b}$, the paths $Z^*_{i,i'}$ and $Z^*_{j,j'}$ are not anticomplete in $G$. Moreover, by the assumption of Lemma~\ref{lem:comp_model_rigid}, the paths $Z^*_{i,1},\ldots,Z^*_{i,b}$ are pairwise anticomplete in $G$, and the paths $Z^*_{j,1},\ldots,Z^*_{j,b}$ are pairwise anticomplete in $G$. But now $(Z^*_{i,1},\ldots,Z^*_{i,b};Z^*_{j,1},\ldots,Z^*_{j,b})$ is a linear induced $(b,b)$-model in $G$. Since $b=\max\{\rho,\sigma\}$, this is a contradiction to the assumption that there is no linear induced $(\rho,\sigma)$-model in $G$. It follows that $I=\varnothing$.

For each $i\in \poi_{\gamma}$, let $Q_i=Z_{i,1}$. Then $z=(Q_1,\ldots, Q_{\gamma})\in \mca{Z}_1\times \cdots \times \mca{Z}_{\gamma}$, and by the above claim, we have $\Phi(z)=I=\varnothing$. From the definition of $\Phi$, it follows that the paths $(Q^*_i:i\in \poi_{\gamma})$ are pairwise anticomplete in $G$. This proves \eqref{st:antiint2}.
\medskip

For every $2$-subset $\{x,y\}$ of $S$, let $Q_{\{x,y\}}\in \mca{Q}_{\{x,y\}}$ be as given by \eqref{st:antiint2}. The last step is to use the choice of $\kappa$ to prove the following (this is similar to the proof of 2.2 in \cite{pinned}).

\sta{\label{st:branchantiint} There is an $s$-subset $S'$ of $S$ such that for all distinct $x,y,z\in S'$, the vertex $x$ is anticomplete to $Q^*_{y,z}$ in $G$.}

Fix an enumeration $x_1,\ldots, x_{\kappa}$ of the vertices in $S$, and  write $Q_{i,j}=Q_{\{x_i,x_j\}}$ for all $i,j\in \poi_{\kappa}$ with $i<j$. For every $3$-subset $T=\{t_1,t_2,t_3\}$ of $\poi_{\kappa}$ with $t_1<t_2<t_3$, let $\Phi(T)$ be the set of all $i\in \{1,2,3\}$ such that $x_{t_i}$ is not anticomplete to $Q^*_{t_j,t_k}$ in $G$, where $\{j,k\}=\{1,2,3\}\setminus \{i\}$ with $j<k$. Then 
$$\Phi:\binom{\poi_{\kappa}}{3}\rightarrow 2^{\{1,2,3\}}$$
is a well-defined map. From Theorem~\ref{thm:multiramsey} and the choice of $\kappa$, we deduce that there is a subset $F\subseteq \{1,2,3\}$ as well as a $\max\{5b,s\}$-subset $Z$ of $\poi_{\kappa}$ such that for every $3$-subset $T$ of $Z$, we have $\Phi(T)=F$. In particular, since $|Z|\geq 5b$, we may choose $I_1,J,I_2,K,I_3\subseteq Z$ with $|I_1|=|J|=|I_2|=|K|=|I_3|=b$ such that
$$\max I_1<\min J\leq \max J<\min I_2\leq \max I_2<\min K\leq \max K<\min I_3.$$

We further claim that $F=\varnothing$. For suppose there exists  $f\in F\subseteq \{1,2,3\}$. Let $J=\{j_1,\ldots, j_{b}\}$ and let $K=\{k_1,\ldots, k_{b}\}$. Then, for every $i\in I_f$ and every $l\in \poi_{b}$, the set $\{i,j_{l},k_{l}\}$ is a $3$-subset of $Z$. It follows that $\Phi(\{i,j_{l},k_{l}\})=F$, and so $f\in \Phi(\{i_l,j_{l'},k_{l'}\})$. This, along with the definition of $\Phi$ and the choice of $I_1,J,I_2,K,I_3$, implies that for  for every $i\in I_f$ and every $l\in \poi_{b}$, the vertex $x_{i}$ is not anticomplete to $Q^*_{j_{l},k_{l}}$ in $G$. Moreover, assuming $I_f=\{i_1,\ldots, i_b\}$, it follows that $\{x_{i_1}, \ldots, x_{i_b}\}$ is a stable set in $G$ because $S$ is, and by \eqref{st:antiint2}, the sets $Q^*_{j_1,k_1}, \ldots, Q^*_{j_{b},k_{b}}$ are pairwise anticomplete in $G$. But now $( \{x_{i_1}\}, \ldots, \{x_{i_b}\}; Q^*_{j_1,k_1}, \ldots, Q^*_{j_{b},k_{b}})$ is a linear induced $(b,b)$-model in $G$, a contradiction. The claim follows.

Choose an $s$-subset $Z'$ of $Z$ (as $|Z|\geq s$), and let $S'=\{x_i:i\in Z'\}\subseteq S$. By the above claim, for every $3$-subset $T=\{t_1,t_2,t_3\}$ of $S'\subseteq S\subseteq \poi_{\kappa}$ with $t_1<t_2<t_3$, we have $\Phi(T)=\varnothing$, and so for all $i\in \{1,2,3\}$, the vertex $x_{t_i}$ is anticomplete to $Q^*_{t_j,t_k}$ in $G$, where $\{j,k\}=\{1,2,3\}\setminus \{i\}$ with $j<k$. In other words, for all distinct $x,y,z\in S'$, the vertex $x$ is anticomplete to $Q^*_{\{y,z\}}$ in $G$. This proves \eqref{st:branchantiint}.
\medskip

Let $S'$ be the $s$-subset of $S$ given by \eqref{st:branchantiint}. Then $S'$ is stable (because $S$ is). Also, by \eqref{st:antiint2} for all distinct $2$-subsets $\{x,y\}, \{x',y'\}$ of $S'\subseteq S$,  the interiors of the paths $Q_{\{x,y\}}$ and $Q_{\{x',y'\}}$ are anticomplete in $G$, and by \eqref{st:branchantiint}, for all distinct $x,y,z\in S'$, the vertex $x$ is anticomplete to $Q^*_{y,z}$ in $G$. Hence,
$$\bigcup_{\{x,y\}\in {S'\choose 2}}Q_{\{x,y\}}$$
is an induced subgraph of $G$ isomorphic to a proper subdivision of $K_{s}$, and so \ref{lem:comp_model_rigid}\ref{lem:comp_model_rigid_a} hold. This completes the proof of Lemma~\ref{lem:comp_model_rigid}.
\end{proof}

Now we can prove Theorem~\ref{thm:comp_model_general}:

\begin{proof}[Proof of Theorem~\ref{thm:comp_model_general}]
Let 
$$\phi=f_{\ref{lem:comp_model_rigid}}(s,t,\rho,\sigma)$$
and let
$$\gamma=g_{\ref{lem:comp_model_rigid}}(s,\rho,\sigma).$$

We will show that 
$$f_{\ref{thm:comp_model_general}}=f_{\ref{thm:comp_model_general}}(s,t,\mu, \rho,\sigma)=f_{\ref{thm:noblock}}(\phi,\mu^{\gamma},t)$$
satisfies the theorem.

Let $G$ be a $t$-clean graph of treewidth more than $f_{\ref{thm:comp_model_general}}$. Suppose that \ref{thm:comp_model_general}\ref{thm:comp_model_general_c} does not hold; that is, there is no linear $\mu$-model in $G$. By Theorem~\ref{thm:noblock}, there is a strong $(\phi,\mu^{\gamma})$-block $(B,\mca{P})$ in $G$. Moreover, we deduce that:

\sta{\label{st:antiint1} For every $2$-subset $\{x,y\}$ of $B$, there is a $\gamma$-subset $\mca{Q}_{\{x,y\}}$ of $\mca{P}_{\{x,y\}}$ such that the paths $(Q^*: Q\in \mca{Q}_{\{x,y\}})$ are pairwise anticomplete in $G$.}

To see this, let $G_{\{x,y\}}$ be the graph with vertex set $\mca{P}_{\{x,y\}}$ such that for all distinct $P,P'\in \mca{P}_{\{x,y\}}$, we have $PP'\in E(G_{\{x,y\}})$ if and only if $P^*$ and $P'^*$ are not anticomplete in $G$. Assume that there are $\mu$ paths $P_1\ldots, P_{\mu}\in \mca{P}_{\{x,y\}}$ such that $\{P_1\ldots, P_{\mu}\}$ is a clique in $G_{\{x,y\}}$. Then $(P^*_1\ldots, P^*_{\mu})$ is a linear $\mu$-model in $G$, contrary to the assumption that \ref{thm:comp_model_general}\ref{thm:comp_model_general_c} does not hold. Thus, there is no clique in $G_{\{x,y\}}$ of cardinality $\mu$. Since $V(G_{\{x,y\}})=|\mca{P}_{\{x,y\}}|\geq \mu^{\gamma}$, it follows from Theorem~\ref{thm:classicalramsey} that there is stable set $\mca{Q}_{\{x,y\}}\subseteq \mca{P}_{\{x,y\}}$ in $G_{\{x,y\}}$ of cardinality $\gamma$. But now by the definition of $G_{\{x,y\}}$, the paths $(Q^*: Q\in \mca{Q}_{\{x,y\}})$ are pairwise anticomplete in $G$. This proves \eqref{st:antiint1}.
\medskip

Henceforth, for every $2$-subset $\{x,y\}$ of $B$, let $\mca{Q}_{\{x,y\}}$ be as given by \eqref{st:antiint1}. It follows that $(B,\mca{Q})$ is a strong $(\phi,\gamma)$-block in $G$ such that for every $\{x,y\}\subseteq B$, the paths $(Q^*: Q\in \mca{Q}_{\{x,y\}})$ are pairwise anticomplete in $G$. By the choice of $\phi,\gamma$, we can apply Lemma~\ref{lem:comp_model_rigid} to $G$ and $(B,\mca{Q})$. But now \ref{lem:comp_model_rigid}\ref{lem:comp_model_rigid_a} implies \ref{thm:comp_model_general}\ref{thm:comp_model_general_a} and \ref{lem:comp_model_rigid}\ref{lem:comp_model_rigid_b} implies \ref{thm:comp_model_general}\ref{thm:comp_model_general_b}. This completes the proof of Theorem~\ref{thm:comp_model_general}.
\end{proof}

From Theorem~\ref{thm:comp_model_general}, we deduce the main result of this section, which we restate:

\compmodel*

\begin{proof}
    Let 
    
    $$f_{\ref{thm:comp_model}}=f_{\ref{thm:comp_model}}(a,b,c,\mu)=f_{\ref{thm:comp_model_general}}(2c+2,\max\{a,b,2c\},b,b,\mu).$$

    Let $G$ be an $(a,b,c)$-candidate of treewidth more than $f_{\ref{thm:comp_model}}$. It is straightforward to check that every subdivision of $W_{2c\times 2c}$ has two anticomplete induced subgraphs each isomorphic to a subdivision of $W_{c\times c}$, and so each with treewidth $c$. It follows in particular that $G$ is $\max\{a,b,2c\}$-clean. Therefore, we can apply Theorem~\ref{thm:comp_model_general} to $G$. Note that every proper subdivision of $K_{2c+2}$ has two anticomplete induced subgraphs each isomorphic to a (proper) subdivision of $K_{c+1}$, and so each with treewidth $c$. Thus, since $G$ is an $(a,b,c)$-candidate, it follows that $G$ has no induced subgraph isomorphic to a proper subdivision of $K_{2c+2}$, and so \ref{thm:comp_model_general}\ref{thm:comp_model_general_a} does not hold. Also, there is no induced $(b,b)$-model in $G$ because $G$ is an $(a,b,c)$-candidate, and so \ref{thm:comp_model_general}\ref{thm:comp_model_general_b} does not hold either. It follows that \ref{thm:comp_model_general}\ref{thm:comp_model_general_c} holds; that is, there is linear $\mu$-model in $G$. This completes the proof of Theorem~\ref{thm:comp_model}.
\end{proof}

\section{Half-induced models} \label{sec:half}
The main result of this section is the following:

\begin{restatable}{theorem}{halfindmain}\label{thm:half_ind_main}
    For all $a,\rho,\sigma\in \poi$, there is a constant $f_{\ref{thm:half_ind_main}}=f_{\ref{thm:half_ind_main}}(a,\rho,\sigma)\in \poi$ with the following property. Let $G$ be a $K_{a}$-free graph and assume that there is a linear $f_{\ref{thm:half_ind_main}}$-model in $G$. Then there is a linear $B$-induced $(\rho, \sigma)$-model in $G$.
\end{restatable}

The proof is in several steps. First, we need to pass to a linear complete model in which every vertex has neighbors in only a few branch sets. The proof is an application of Theorem~\ref{thm:multiramsey}, similar to the proof of \eqref{st:branchantiint} from Theorem \ref{thm:comp_model}:

\begin{lemma}\label{lem:half_ind0}
    For all $a,\mu,\rho,\sigma\in \poi$, there is a constant $f_{\ref{lem:half_ind0}}=f_{\ref{lem:half_ind0}}(a,\mu,\rho,\sigma)\in \poi$ with the following property. Let $G$ be a $K_{a}$-free graph and let $K^0$ be a linear $f_{\ref{thm:half_ind_main}}$-model in $G$. Then one of the following holds.
    \begin{enumerate}[{\rm (a)}]
\item\label{lem:half_ind0_a} There is a linear $B$-induced $(\rho,\sigma)$-model in $G$.
\item\label{lem:half_ind0_b} There is a linear $\mu$-model $K$ in $G$ such that each branch set of $K$ is a branch set of $K^0$, and every vertex in $V(K)$ has neighbors in at most $\rho$ branch sets of $K$.
    \end{enumerate}
\end{lemma}

\begin{proof}

Let $$f_{\ref{lem:half_ind0}}=f_{\ref{lem:half_ind0}}(a,\mu,\rho,\sigma)=f_{\ref{thm:multiramsey}}(2^{\rho+1},\rho+1,\max\{a^{\sigma}(\rho+1)+\rho,\mu\}).$$ 

Let $G$ be a $K_{a}$-free graph and let $K^0=(C_1,\ldots, C_{f_{\ref{lem:half_ind0}}})$ be a linear $f_{\ref{lem:half_ind0}}$-model in $G$. Suppose that there is no linear $B$-induced $(\rho,\sigma)$-model in $G$. 

For every $(\rho+1)$-subset $S=\{s_1,\ldots,s_{\rho+1}\}$ of $\poi_{f_{\ref{lem:half_ind0}}}$ with $s_1<\ldots<s_{\rho+1}$, let $\Phi(S)$ be the set of all $i\in \poi_{\rho+1}$ such that some vertex in $C_{s_i}$ has neighbors in $C_{s_j}$ for every $j\in \poi_{\rho+1}\setminus \{i\}$. Then 
$$\Phi:\binom{\poi_{f_{\ref{lem:half_ind0}}}}{\rho+1}\rightarrow 2^{\poi_{\rho+1}}$$
is a well-defined map. From Theorem~\ref{thm:multiramsey} and the choice of $f_{\ref{lem:half_ind0}}$, we deduce that there is a subset $F\subseteq \poi_{\rho+1}$ as well as a $\max\{a^{\sigma}(\rho+1)+\rho,\mu\}$-subset $Z$ of $\poi_{f_{\ref{lem:half_ind0}}}$ such that for every $(\rho+1)$-subset $S$ of $Z$, we have $\Phi(S)=F$. In particular, since $|Z|\geq a^{\sigma}(\rho+1)+\rho$, we may choose $I_1,\ldots,I_{\rho+1}\subseteq Z$ and $j_1,\ldots, j_{\rho}\in Z$ such that \begin{itemize}
    \item $|I_1|=\cdots=|I_{\rho+1}|=a^{\sigma}$; and
\item for every $h\in \poi_{\rho}$, we have $\max I_h<j_h<\min I_{h+1}$.
\end{itemize}

We further claim that:

\sta{\label{st:Fempty}$F=\varnothing$.}

Suppose for a contradiction that there exists  $f\in F\subseteq \poi_{\rho+1}$. Then, for every $i\in I_f$, the set $\{i,j_1,\ldots, j_{\rho}\}$ is a $(\rho+1)$-subset of $Z$. It follows that $\Phi(\{i,j_1,\ldots, j_{\rho}\})=F$, and so $f\in \Phi(\{i,j_1,\ldots, j_{\rho}\})$. This, along with the definition of $\Phi$ and the choice of $I_1,\ldots,I_{\rho+1}\subseteq Z$ and $j_1,\ldots, j_{\rho}\in Z$, implies that for every $i\in I_f$, there is a vertex $x_{i}\in C_i$ such that $x_i$ has neighbors in $C_{j_h}$ for every $h\in \poi_{\rho}$. Moreover, since $|I_f|=a^{\sigma}$ and since $G$ is $K_a$-free, it follows that there are $i_1,\ldots, i_{\sigma}\in I_f$ such that $\{x_{i_1},\ldots, x_{i_{\sigma}}\}$ is a stable set of cardinality $\sigma$ in $G$. But now $(C_{j_1}, \ldots,C_{j_{\rho}}; \{x_{i}\}, \ldots, \{x_{i_{\sigma}}\})$ is a linear $B$-induced $(\rho,\sigma)$-model in $G$, a contradiction. This proves \eqref{st:Fempty}.
\medskip

Choose a $\mu$-subset $I$ of $Z$ (this is possible as $|Z|\geq \mu$). By \eqref{st:Fempty}, for every $(\rho+1)$-subset $S$ of $I\subseteq Z\subseteq \poi_{f_{\ref{lem:half_ind0}}}$, we have $\Phi(S)=\varnothing$. In other words, for every $i\in I$, every vertex in $C_i$ has neighbors in fewer that $\rho$ sets among $(C_j:j\in I\setminus \{i\})$. But now $K=(C_j:j\in I)$ is a linear $\mu$-model in $G$ that satisfies \ref{lem:half_ind0}\ref{lem:half_ind0_b}. This completes the proof of Lemma~\ref{lem:half_ind0}.
\end{proof}

The second step is to show that if every vertex of the linear complete model has neighbors in only a few branch sets, then the conclusion of Theorem~\ref{thm:half_ind_main} holds. For the purpose of an inductive argument, we need to prove a technical strengthening:

\begin{restatable}{lemma}{halfindtwo}\label{lem:half_ind2}
    For all $\varrho,\rho,\varsigma,\sigma\in \poi$, there is a constant $f_{\ref{lem:half_ind2}}=f_{\ref{lem:half_ind2}}(\varrho,\rho,\varsigma,\sigma)\in \poi$ with the following property. Let $G$ be a graph and let $K=(C_1,\ldots, C_{f_{\ref{lem:half_ind2}}})$ be a linear $f_{\ref{lem:half_ind2}}$-model in $G$ such that every vertex in $V(K)$ has neighbors in at most $\rho$ branch sets of $K$. Then one of the following holds.
    
     \begin{enumerate}[{\rm (a)}]
     \item\label{lem:half_ind2_a} There is a linear $B$-induced $(\rho, \sigma)$-model in $G$.
        \item\label{lem:half_ind2_b} \sloppy There is a $B$-induced $(\varrho,\varsigma)$-model $N=(C_{j_1},\ldots, C_{j_{\varrho}}; B_{k_1},\ldots, B_{k_{\varsigma}})$ in $G$ where $j_1,\ldots,j_{\varrho},k_1,\ldots,k_{\varsigma}\in \poi_{f_{\ref{lem:half_ind2}}}$ with $j_1<\cdots <j_{\varrho}<k_1<\cdots<k_{\varsigma}$, and $B_{k_h}\subseteq C_{k_h}$ for every $h\in \poi_{\varsigma}$. In particular, $N$ is linear.
        \end{enumerate}
\end{restatable}

The proof of Lemma \ref{lem:half_ind2} relies on yet another result, Lemma \ref{lem:half_ind1} below, and that lemma needs a new definition, that of an ``aligned'' model. Let $\rho,\sigma\in \poi$, let $G$ be a graph and let $M=(A_1,\ldots, A_{\rho}; B_1,\ldots, B_{\sigma})$ be a $(\rho,\sigma)$-model in $G$. We say that $M$ is $A$-\textit{aligned} if $M$ is $B$-linear and for each $j\in [\sigma]$, there are $\rho$ pairwise disjoint subpaths $(u^j_i\dd B_j\dd v^j_i:i\in \poi_{\rho})$ of $B_j$ (possibly of length zero) such that 
 \begin{enumerate}[{\rm ({A}L1)}]
        \item\label{AL1} there is an end $u_j$ of the path $B_j$ such that $B_j$ traverses $u_j,u^j_1,v^j_1,\ldots, u^j_{\rho},v^j_{\rho}$ in this order; and
        \item\label{AL2} for every $i\in \poi_{\rho}$, all vertices in $B_j$ with a neighbor in $A_i$ are contained in $u^j_i\dd B_j\dd v^j_i$.
        \end{enumerate}

We say that $M$ is \textit{$B$-aligned} if $M^{\T}$ is $A$-aligned.
\medskip

The following is a key tool in both the proof of Lemma~\ref{lem:half_ind2} and the completion of the proof of Theorem~\ref{thm:EEZ_candidate} in the next section. Similar result have also appeared in our earlier papers \cite{tw9,tw11,tw12,ti2} based on the general idea that when we consider how different sets attach to the same path, we generally either get an ``aligned'' outcome or can split the path into several subpaths, each attaching to each set (which gave us a ``constellation'' in earlier papers, and here gives us a half-induced model). 

\begin{lemma}\label{lem:half_ind1}
    Let $d,\alpha,\rho,\sigma\in \poi$ and let $\theta=\alpha^{2(\sigma-1)}(d(\sigma-1)+\rho)$. Let $G$ be a graph and let $M=(A_1,\ldots, A_{\theta}; B)$ be a $B$-linear  $(\theta, 1)$-model in $G$ such that every vertex in $B$ has neighbors in at most $d$ sets among $A_1,\ldots, A_{\theta}$. Then one of the following holds.
     \begin{enumerate}[{\rm (a)}]
        \item \label{lem:half_ind1_a} There is an $A$-aligned $(\alpha,1)$-model $M'=(A_{i_1},\ldots, A_{i_{\alpha}}; B')$ in $G$ where $i_1,\ldots,i_{\alpha}\in \poi_\theta$ with $i_1<\cdots <i_{\alpha}$, and $B'\subseteq B$.
        \item\label{lem:half_ind1_b} There is a $B$-induced $(\rho, \sigma)$-model $M''=(A''_{1},\ldots, A''_{\rho}; B''_{1},\ldots, B''_{\sigma})$ in $G$ such that $A''_{1},\ldots, A''_{\rho}\in \{A_1,\ldots, A_{\theta}\}$ and $B(M'')\subseteq B$.
        \end{enumerate}
\end{lemma}
\begin{proof}The proof is by induction on $\sigma$ for fixed $d,\alpha$ and $\rho$. Suppose that $\sigma=1$. Then $\theta=\rho$ and $M''=M$ is a $B$-induced $(\rho,\sigma)$-model in $G$ which satisfies \ref{lem:half_ind1}\ref{lem:half_ind1_b}. So we may assume that $\sigma\geq 2$.

Let $u$ and $v$ be the ends of $B$. For every $i\in \poi_{\theta}$, traversing $B$ from $u$ to $v$, let $u_i$ and $v_i$, respectively, be the first and the last vertex in $B$ with a neighbor in $A_i$, and let $P_i=u_i\dd B\dd v_i$. Define $\Gamma$ to be the graph with vertex set $\poi_{\theta}$ such that for all distinct $i,i'\in \poi_{\theta}$, we have $ii'\in E(\Gamma)$ if and only if $P_i\cap P_{i'}\neq\varnothing$. Then $\Gamma$ is an interval graph, and so $\Gamma$ is perfect \cite{berge}. Also, assuming $\theta'=\alpha^{2(\sigma-2)}(d(\sigma-2)+\rho)$, it follows from $\sigma\geq 2$ that
$$|V(\Gamma)|=\theta=\alpha^{2(\sigma-1)}(d(\sigma-1)+\rho)
      = \alpha^2\theta'+\alpha^{2(\sigma-1)}d
      \geq \alpha^2(\theta'+d).$$
 In summary, $\Gamma$ is a perfect graph with $|V(\Gamma)|\geq \alpha^2(\theta'+d)$, and so $\Gamma$ contains either a stable set $I_1$ of cardinality $\alpha^2$ or a clique $I_2$ of cardinality $\theta'+d$.
 
First, assume that there is a stable set $I_1$ in $\Gamma$ with $|I_1|=\alpha^2$. It follows that the paths $(P_i:i\in I_1)$ are pairwise disjoint. Since $|I_1|=\alpha^2$ and by the Erd\H{o}s-Szekeres Theorem \cite{ErdSze}, there are $i_1,\ldots, i_{\alpha}\in I_1\subseteq \poi_{\theta}$ with $i_1<\cdots< i_{\alpha}$ such that $B$ traverses either
$$u,u_{i_1},v_{i_1}, \ldots, u_{i_{\alpha}},v_{i_{\alpha}}$$
or  
$$v,v_{i_1},u_{i_1}, \ldots, v_{i_{\alpha}},u_{i_{\alpha}}$$
in this order. But now by definition, $M'=(A_{i_1},\ldots, A_{i_{\alpha}}; B)$ is an $A$-aligned $(\alpha, 1)$-model in $G$ that satisfies \ref{lem:half_ind1}\ref{lem:half_ind1_a}.

Second, assume that $\Gamma$ contains a clique $I_2$ of cardinality $\theta'+d$. Then there is a vertex $x\in B$ such that for every $i\in I_2\subseteq \poi_{\theta}$, we have $x\in P_i$. On the other hand, by the assumption of \ref{lem:half_ind1}, there are at most $d$ values of $i\in I_2\subseteq \poi_{\theta}$ for which $x$ has neighbors in $A_i$. It follows $B\setminus \{x\}$ has two components, say $L$ and $R$, and there is a $\theta'$-subset $J$ of $I_2$ such that for every $j\in J$, neither $L$ nor $R$ is anticomplete to $A_i$.

Let $J=\{j_1,\ldots, j_{\theta'}\}$ with $j_1<\cdots <j_{\theta'}$. Then both $M_L=(A_{j_1},\ldots, A_{j_{\theta'}}; L)$ and $M_R=(A_{j_1},\ldots, A_{j_{\theta'}}; R)$ are $B$-linear $(\theta',1)$-models in $G$. Moreover, every vertex in $L\subseteq B$ has neighbors in at most $d$ sets among $A_{j_1},\ldots, A_{j_{\theta'}}$. From the choice of $\theta'$ and the inductive hypothesis applied to $M_L$, we deduce that one of the following holds.
\begin{itemize}
    \item There is an $A$-aligned $(\alpha,1)$-model $M'=(A_{i_1},\ldots, A_{i_{\alpha}}; B')$ in $G$ where $i_1,\ldots, i_{\alpha}\in J$ with $i_1<\cdots <i_{\alpha}$, and $B'\subseteq L$.
        \item There is a $B$-induced $(\rho, \sigma-1)$-model $M^-=(A''_{1},\ldots, A''_{\rho}; B''_{1},\ldots, B''_{\sigma-1})$ in $G$ such that $A''_{1},\ldots, A''_{\rho}\in \{A_{j_1},\ldots, A_{j_{\theta'}}\}$ and $B(M^-)\subseteq L$.
\end{itemize}

In the former case, $M'$ satisfies \ref{lem:half_ind1}\ref{lem:half_ind1_a} because $J\subseteq \poi_{\theta}$ and $L\subseteq B$. In the latter case, let $B''_{\sigma}=R$. Since $B(M^-)\subseteq L$ and since $L$ and $R$ are anticomplete in $G$, it follows that $M''=(A''_{1},\ldots, A''_{\rho}; B''_{1},\ldots, B''_{\sigma-1}, B''_{\sigma})$ is a $B$-induced $(\rho, \sigma)$-model in $G$ such that $A''_{1},\ldots, A''_{\rho}\in \{A_{j_1},\ldots, A_{j_{\theta'}}\}\subseteq \{A_1,\ldots, A_{\theta}\}$ and $B(M'')=B(M^-)\cup R\subseteq L\cup R\subseteq B$. But now $M''$ satisfies \ref{lem:half_ind1}\ref{lem:half_ind1_b}. This completes the proof of Lemma~\ref{lem:half_ind1}.
  \end{proof}

Now we give a proof of Lemma~\ref{lem:half_ind2}, which we also restate:

\halfindtwo*

\begin{proof}
Let $\varrho,\rho,\sigma\in \poi$ (meaning all variables in the statement except for $\varsigma$) be fixed. First, we define a sequence $\{\mu_{\varsigma}\}_{\varsigma\in \poi}$ of positive integers recursively, as follow. Let $\mu_1=\varrho$ and for $\varsigma\geq 2$, let 
$$\mu_{\varsigma}=(\mu_{\varsigma-1}+1)^{2(\sigma-1)}(\rho(\sigma-1)+\rho).$$
This concludes the recursive definition of $\{\mu_{\varsigma}\}_{\varsigma\in \poi}$. 

Back to the proof of \ref{lem:half_ind2}, we will show by induction on $\varsigma$ that:
$$f_{\ref{lem:half_ind2}}=f_{\ref{lem:half_ind2}}(\varrho,\rho,\varsigma,\sigma)=\mu_{\varsigma}+1$$
satisfies the lemma.

Let $G$ be a graph and let $K=(C_1,\ldots, C_{\mu_{\varsigma}+1})$ be a linear $(\mu_{\varsigma}+1)$-model in $G$ such that every vertex in $V(K)$ has neighbors in at most $\rho$ branch sets of $K$. Assume that \ref{lem:half_ind2}\ref{lem:half_ind2_a} does not hold; that is, there is no linear $B$-induced $(\rho,\sigma)$-model in $G$.
\medskip

Let 
$$M=(C_1,C_2\ldots, C_{\mu_{\varsigma}}; C_{\mu_{\varsigma}+1}).$$ Then $M$ is a linear $(\mu_{\varsigma},1)$-model in $G$. In particular, if $\varsigma=1$, then $\mu_{\varsigma}=\varrho$ and $N=M$ satisfies \ref{lem:half_ind2}\ref{lem:half_ind2_b}. So we may assume that $\varsigma\geq 2$. We further claim that:

\sta{\label{st:getaligned} There is an $A$-aligned $(\mu_{\varsigma-1}+1,1)$-model $M'=(C_{i_1},\ldots, C_{i_{\mu_{\varsigma-1}+1}}; B')$ in $G$ where $1\leq i_1<\cdots <i_{\mu_{\varsigma-1}+1}\leq \mu_{\varsigma}$ and $B'\subseteq C_{\mu_{\varsigma}+1}$. }

Since every vertex in $C_{\mu_{\varsigma}+1}$ has neighbors in at most $\rho$ sets among $C_1,C_2\ldots, C_{\mu_{\varsigma}}$, and by the choice of $\mu_{\varsigma}$, we can apply Lemma~\ref{lem:half_ind1} to $M$. Observe that \ref{lem:half_ind1}\ref{lem:half_ind1_a} yields exactly the assertion of \eqref{st:getaligned}. So we may assume that \ref{lem:half_ind1}\ref{lem:half_ind1_b} holds; that is, there is a $B$-induced $(\rho, \sigma)$-model $M''=(A''_{1},\ldots, A''_{\rho}; B''_{1},\ldots, B''_{\sigma})$ in $G$ such that $A''_{1},\ldots, A''_{\rho}\in \{C_1,\ldots, C_{\mu_{\varsigma}}\}$ and $B(M'')\subseteq C_{\mu_{\varsigma}+1}$. In particular, $M''$ is linear because $M$ is. But now $M''$ is a linear $B$-induced $(\rho,\sigma)$-model in $G$, a contradiction. This proves \eqref{st:getaligned}.
\medskip

Henceforth, let $M'$ be as given by \eqref{st:getaligned}. Recall that $K$ is a linear $(\mu_{\varsigma}+1)$-model in $G$ and every vertex in $V(K)$ has neighbors in at most $\rho$ branch sets of $K$. Since $C_{i_1},\ldots, C_{i_{\mu_{\varsigma-1}+1}}$ are branch sets of $K$, it follows that $K^-=(C_{i_1},\ldots, C_{i_{\mu_{\varsigma-1}+1}})$ is a linear $(\mu_{\varsigma-1}+1)$-model in $G$ and every vertex in $V(K^-)$ has neighbors in at most $\rho$ branch sets of $K^-$.  Therefore, we can apply the inductive hypothesis to $K^-$. Since there is no linear $B$-induced $(\rho, \sigma)$-model in $G$, it follows that \ref{lem:half_ind2}\ref{lem:half_ind2_b} holds. Explicitly, we have:

\sta{\label{st:indhyp} There is a $B$-induced $(\varrho,\varsigma-1)$-model $N^-=(C_{j_1},\ldots, C_{j_{\varrho}}; B_{k_1},\ldots, B_{k_{\varsigma-1}})$ in $G$ where $j_1,\ldots,j_{\varrho},k_1,\ldots,k_{\varsigma-1}\in \{i_1,\ldots, i_{\mu_{\varsigma-1}+1}\}$ with $j_1<\cdots <j_{\varrho}<k_1<\cdots<k_{\varsigma}$, and $B_{k_h}\subseteq C_{k_h}$ for every $h\in \poi_{\varsigma-1}$. In particular, $N^-$ is linear.}

Henceforth, let $N^-$ be as given by \eqref{st:indhyp}. Recall that $M'=(C_{i_1},\ldots, C_{i_{\mu_{\varsigma-1}+1}},B')$, as given by \eqref{st:getaligned}, is $A$-aligned. Since $j_1,\ldots,j_{\varrho},k_1,\ldots,k_{\varsigma-1}\in \{i_1,\ldots, i_{\mu_{\varsigma-1}+1}\}$ with $j_1<\cdots <j_{\varrho}<k_1<\cdots<k_{\varsigma-1}$, it follows that the $(\varrho+\varsigma-1,1)$-model $(C_{j_1},\ldots, C_{j_{\varrho}}, C_{k_1},\ldots, C_{k_{\varsigma-1}}; B')$ in $G$ is also $A$-aligned. In particular, there are $\varrho+\varsigma-1$ pairwise disjoint subpaths $$(u_h\dd B'\dd v_h:h\in \{j_1,\ldots,j_{\varrho},k_1,\ldots,k_{\varsigma-1}\})$$
of $B'$ (possibly of length zero) such that 
 \begin{itemize}
        \item there is an end $u$ of the path $B'$ such that $B'$ traverses $$u,u_{j_1},v_{j_1},\ldots, u_{j_{\varrho}},v_{j_{\varrho}},u_{k_1},v_{k_1},\ldots, u_{k_{\varsigma-1}},v_{k_{\varsigma-1}}$$
        in this order; and
        \item for every $h\in \{j_1,\ldots,j_{\varrho},k_1,\ldots,k_{\varsigma-1}\}$, all vertices in $B'$ with a neighbor in $C_h$ are contained in $u_h\dd B'\dd v_h$.
        \end{itemize}
        
Let $B=u\dd B'\dd v_{j_{\varrho}}$. It follows that for every $h\in \poi_{\varrho}$, the paths $B$ and $C_{j_h}$ are not anticomplete in $G$, while for every $h\in \poi_{\varsigma-1}$, the paths $B$ and $C_{k_h}$ are anticomplete in $G$. On the other hand, $N^-=(C_{j_1},\ldots, C_{j_{\varrho}}; B_{k_1},\ldots, B_{k_{\varsigma-1}})$, as given by \eqref{st:indhyp}, is a linear $B$-induced $(\varrho,\varsigma-1)$-model in $G$ with $B_{k_h}\subseteq C_{k_h}$ for every $h\in \poi_{\varsigma-1}$. 

Let  $k_{\varsigma}=\mu_{\varsigma}+1$ and let $B_{k_{\varsigma}}=B$. We deduce that
$$N=(C_{j_1},\ldots, C_{j_{\varrho}}; B_{k_1},\ldots, B_{k_{\varsigma-1}},B_{k_{\varsigma}})$$
is a linear $B$-induced $(\varrho,\varsigma)$-model in $G$. But now $N$ satisfies \ref{lem:half_ind2}\ref{lem:half_ind2_b} because, from \eqref{st:getaligned}, \eqref{st:indhyp}, and the choice of $k_{\varsigma}$ and $B_{k_{\varsigma}}$, it follows that 
\begin{itemize}
    \item $1\leq j_1<\cdots <j_{\varrho}<k_1<\cdots<k_{\varsigma-1}\leq \mu_{\varsigma}<\mu_{\varsigma}+1=k_{\varsigma}=f_{\ref{lem:half_ind2}}$;
    \item $B_{k_h}\subseteq C_{k_h}$ for every $h\in \poi_{\varsigma-1}$; and
    \item $B_{k_{\varsigma}}=B\subseteq B'\subseteq C_{\mu_{\varsigma}+1}=C_{k_{\varsigma}}$.
\end{itemize}
 This completes the proof of Lemma~\ref{lem:half_ind2}.
\end{proof}

With Lemmas~\ref{lem:half_ind0} and \ref{lem:half_ind2} in our arsenal, Theorem~\ref{thm:half_ind_main} is now almost immediate:

\halfindmain*
\begin{proof}
    Let 
    $$\mu=f_{\ref{lem:half_ind2}}(\rho,\rho,\sigma,\sigma)$$ 
    and let $$f_{\ref{thm:half_ind_main}}=f_{\ref{thm:half_ind_main}}(a,\rho,\sigma)=f_{\ref{lem:half_ind0}}(a,\mu,\rho,\sigma).$$ 

Let $G$ be a $K_a$-free graph and let $K^0$ be a linear $f_{\ref{thm:half_ind_main}}$-model in $G$. Apply Lemma~\ref{lem:half_ind0} to $K^0$. If \ref{lem:half_ind0}\ref{lem:half_ind0_b} holds, then there is a linear $B$-induced $(\rho,\sigma)$-model in $G$, as desired. So we may assume that \ref{lem:half_ind0}\ref{lem:half_ind0_a} holds; that is, there is a linear $\mu$-model $K$ in $G$ such that every vertex in $V(K)$ has neighbors in at most $\rho$ branch sets of $K$. But now from Lemma~\ref{lem:half_ind2} applied to $K$, and by the choice of $\mu$, we deduce that there is a $B$-induced $(\rho,\sigma)$-model in $G$. This completes the proof of Theorem~\ref{thm:half_ind_main}.
\end{proof}

\section{From half induced to induced}\label{sec:halftofull}

We are now ready to prove Theorem~\ref{thm:EEZ_candidate} (which, as pointed out in Section~\ref{sec:model}, is equivalent to Theorem~\ref{thm:EEZ_ind_minor}):

\candidate*
\begin{proof}
Let

$$\beta=f_{\ref{thm:multiramsey}}(2^{c^2},2,2b);$$

$$\alpha=\beta(c+1);$$

$$\theta=\alpha^{2(b-1)}b^2;$$

$$\gamma=a^b\binom{\theta}{b}+c\binom{\theta}{\alpha};$$

$$\mu=f_{\ref{thm:half_ind_main}}(a,\gamma,\theta).$$

We claim that
$$f_{\ref{thm:EEZ_candidate}}=f_{\ref{thm:EEZ_candidate}}(a,b,c)=f_{\ref{thm:comp_model}}(a,b,c,\mu)$$
satisfies the theorem.

Suppose for a contradiction that there is a $(a,b,c)$-candidate $G$ with $\tw(G)>f_{\ref{thm:EEZ_candidate}}$. By Theorem~\ref{thm:comp_model}, there is a linear $\mu$-model in $G$. By Theorem~\ref{thm:half_ind_main} and the choice of $\mu$, there is a linear $B$-induced $(\gamma,\theta)$-model in $G$. It follows that there is a linear $A$-induced $(\theta,\gamma)$-model $M$ in $G$ (though the reader may notice we will never use the fact that $M$ is $A$-linear).

Let $A_1,\dots, A_{\theta}$ be the branch sets of $M$ contained in $A(M)$. Then $A_1,\dots, A_{\theta}$ are pairwise anticomplete in $G$. Let $\mca{B}$ be the set of all branch sets of $M$ contained in $B(M)$ (so we have $|\mca{B}|=\gamma$), and let $\mca{B}'$ be the set of all $B\in \mca{B}$ for which every vertex in $B$ has neighbors in fewer than $b$ sets among $A_1,\dots, A_{\theta}$. We claim that:

\sta{\label{st:BseesfewAs} $|\mca{B}'|\geq c\binom{\theta}{\alpha}$.}

Suppose not. Since $|\mca{B}|=\gamma=a^b\binom{\theta}{b}+c\binom{\theta}{\alpha}$, it follows that $|\mca{B}\setminus \mca{B}'|\geq  a^b\binom{\theta}{b}$. By definition, for every $B\in \mca{B}\setminus \mca{B}'$, there is a vertex $x_B\in B$ and a $b$-subset $I_B$ of $\poi_{\theta}$ such that $x_B$ has neighbors in $A_i$ for every $i\in I_B$. Since $|\mca{B}\setminus \mca{B}'|\geq a^{b}\binom{\theta}{b}$, it follows that there is a $b$-subset $I$ of $\poi_{\theta}$ as well as a $a^{b}$-subset $\mca{T}$ of $\mca{B}\setminus \mca{B}'$ such that for every $B\in \mca{T}$, we have $I_B=I$. Moreover, since $G$ is $K_{a}$-free and since $\mca{T}=a^{b}$, it follows from Theorem~\ref{thm:classicalramsey} that there is a $b$-subset $\mca{S}$ of $\mca{T}$ such that $\{x_B:B\in \mca{T}\}$ is a stable set of cardinality $b$ in $G$. But then $(A_i:i\in I \ ;\  \{x_B\}: B\in b)$ is an induced $(b,b)$-model in $G$, contrary to the assumption that $G$ is an $(a,b,c)$-candidate. This proves \eqref{st:BseesfewAs}.

\sta{\label{st:aligned} Let $B\in \mca{B}'$. Then there are $i_{1,B},\ldots,i_{\alpha,B}\in \poi_\theta$ with $i_{1,B}<\cdots <i_{\alpha,B}$, and $B'\subseteq B$, such that $(A_{i_{1,B}},\ldots, A_{i_{\alpha,B}}; B')$ is an $A$-aligned $(\alpha,1)$-model in $G$.}

To see this, note that $M_B=(A_1,\ldots, A_{\theta},B)$ is a $B$-linear $(\theta,1)$-model in $G$ (which is also $A$-induced). Moreover, by the definition of $\mca{B}'$, every vertex in $B$ has neighbors in fewer that $b$ sets among $A_1,\ldots, A_{\theta}$. Recall also that $\theta=\alpha^{2(b-1)}(b(b-1)+b)$. Therefore, we can apply Lemma~\ref{lem:half_ind1} to $M_B$. Observe that \ref{lem:half_ind1}\ref{lem:half_ind1_a} yields exactly the assertion of \eqref{st:aligned}. So we may assume that \ref{lem:half_ind1}\ref{lem:half_ind1_b} holds; in particular, there is a $B$-induced $(b, b)$-model $M''=(A''_{1},\ldots, A''_{b}; B''_{1},\ldots, B''_{b})$ in $G$ such that $A''_{1},\ldots, A''_{b}\in \{A_{1},\ldots, A_{\theta}\}$. It follows that $M''$ is $A$-induced, as well, because $M_B$ is. But now $M''$ is an induced $(b,b)$-model in $G$, a contradiction. This proves \eqref{st:aligned}.
\medskip

Combining \eqref{st:BseesfewAs} and \eqref{st:aligned}, we deduce that there are 
\begin{itemize}
    \item $i_1,\ldots,i_{\alpha}\in \poi_\theta$ with $i_1<\cdots <i_{\alpha}$;
    \item $B_1,\ldots, B_{c}\in \mca{B}'$; and
    \item $B'_j\subseteq B_j$ for every $j\in \poi_{c}$
\end{itemize}
such that $N=(A_{i_1},\ldots, A_{i_{\alpha}}; B'_1,\ldots, B'_{c})$ is an $A$-aligned $(\alpha,c)$-model in $G$. Recall that $\alpha=\beta(c+1)$. For every $j\in \poi_{\beta}$ and $k\in \poi_{c+1}$, we write 
$$A_{j,k}=A_{i_{(c+1)(j-1)+k}}.$$

Let $l\in \poi_{c}$ be fixed. Since $N$ is $A$-aligned, it follows that there are $\beta$ pairwise disjoint subpaths $(P^l_j=u^l_{j}\dd B'_l\dd v^l_{j}:j\in \poi_{\beta})$ of $B'_l$ (possibly of length zero) such that: 
\begin{itemize}
        \item there is an end $u_l$ of the path $B'_l$ such that $B'_l$ traverses $u_l,u^l_{1},v^l_{1},\ldots, u^l_{\beta},v^l_{\beta}$ in this order; and
        \item for every $j\in \poi_{\beta}$, all vertices in $B'_l$ with a neighbor in $\bigcup_{k=1}^{c}A_{j,k}$ are contained in $P^l_j$.
        \end{itemize}
Note that the paths $(P^l_j=u^l_{j}\dd B'_l\dd v^l_{j}:j\in \poi_{\beta})$ are pairwise anticomplete, because $B'_l$ contains a neighbor of $A_{j, c+1}$ between $v_j^l$ and $u_{j+1}^l$ for every $j \in \mathbb{N}_{\beta-1}$.

For every $j\in \poi_{\beta}$, let 
$$M_j=(A_{j,1},\ldots, A_{j,c} ; P^1_j,\ldots, P^{c}_j).$$
Then $M_j$ is a $(c,c)$-model in $G$ which is linear and $A$-induced. Moreover, for all distinct $j,j'\in \poi_{\beta}$, since $N$ is $A$-induced and $A$-aligned, the sets $A(M_j)$ and $A(M_{j'})\cup B(M_{j'})$ are anticomplete in $G$.
\medskip

Now, for every $2$-subset $\{j,j'\}$ of $\poi_{\beta}$ with $j < j'$, let $\Phi(\{j,j'\})$ be the set of all $(l,l')\in \poi_{c}\times \poi_{c}$ such that $P^l_j$ and $P^{l'}_{j'}$ are not anticomplete in $G$; in particular, we have $(l,l')\in \Phi(\{j,j'\})$ only if $l\neq l'$. It follows that 
$$\Phi: \binom{\poi_{\beta}}{2}\rightarrow 2^{\poi_{c}\times \poi_{c}}$$
is a well-defined map. From Theorem~\ref{thm:multiramsey} and the choice of $\beta$, we deduce that there is a subset $F\subseteq \poi_{c}\times \poi_{c}$ as well as a $2b$-subset $Z$ of $\poi_{\beta}$ such that for every $2$-subset $\{j,j'\}$ of $Z$, we have $\Phi(\{j,j'\})=F$. Let us write 
$Z=\{j_1,\ldots, j_{b},j'_{1},\ldots, j'_{b}\}$
where $j_1<\cdots<j_{b}<j'_{1}<\cdots< j'_{b}$.

Assume that $F\neq \varnothing$; say $(l,l')\in F$. Then, for all $k,k'\in \poi_{b}$, we have $(l,l')\in F=\Phi(\{j_{k},j'_{k'}\})$. It follows that $l\neq l'$, and for all $k,k'\in \poi_{b}$, the paths $P^l_{j_k}$ and $P^{l'}_{j'_{k'}}$ are not anticomplete in $G$. But then 
$$(P^l_{j_1}, \ldots, P^l_{j_{b}}; P^{l'}_{j'_{1}}, \ldots, P^{l'}_{j'_{b}})$$
is a (linear) induced $(b,b)$-model in $G$, a contradiction.

We deduce that $F=\varnothing$. In particular, we have $\Phi(\{j_1, j'_1\})=F=\varnothing$. From the definition of $\Phi$, it follows that for all $l,l'\in \poi_{c}$, the paths $P^l_{j_1}$ and $P^{l'}_{j'_1}$ are anticomplete in $G$. In other words, $B(M_{j_1})$ and $B(M_{j'_1})$ are anticomplete in $G$, and so $X=A(M_{j_1})\cup B(M_{j_1})$ and $Y=A(M_{j'_1})\cup B(M_{j'_1})$ are anticomplete in $G$. But now $X$ and $Y$ are anticomplete subsets of $V(G)$, each inducing a subgraph of treewidth at least $c$ (because that subgraph has a minor isomorphic to $K_{c,c}$),  contrary to the assumption that $G$ is an $(a,b,c)$-candidate. This completes the proof of Theorem~\ref{thm:EEZ_candidate}.
\end{proof}

\section{From induced minors to induced subgraphs} \label{sec:isg}
In this section, we deduce Theorem~\ref{thm:EEZ_isg} from Theorem~\ref{thm:EEZ_candidate}. We will also use the main result of \cite{tw16} which involves a new set of definitions.
\medskip

For a set $X$, a linear order $\pre$ on $X$, and $x,y\in X$, we write $x\prec y$ to mean $x\pre y$ and $x$ and $y$ are distinct. For an element $x\in X$ and a subset $Y\subseteq X$, we write $x\prec Y$ to mean $x\prec y$ for every $y\in Y$. Similarly, we write $Y\prec x$ to mean $y\prec x$ for every $y\in Y$.

Let $l,s\in \poi$. An \textit{$(s,l)$-constellation} is a graph $\mf{c}$ in which there is a stable set $S_{\mf{c}}$ of cardinality $s$ such that $\mf{c}\setminus S_{\mf{c}}$ has exactly $l$ components, every component of $\mf{c}\setminus S_{\mf{c}}$ is a path, and each vertex $x\in S_{\mf{c}}$ has at least one neighbor in each component of $\mf{c}\setminus S_{\mf{c}}$. We denote by $\mca{L}_{\mf{c}}$ the set of all components $\mf{c}\setminus S_{\mf{c}}$ (each of which is a path), and also denote $\mf{c}$ by the pair $(S_{\mf{c}},\mca{L}_{\mf{c}})$.When $l=1$, say $\mca{L}_{\mf{c}}=\{L_{\mf{c}}\}$, we call $\mf{c}$ an \textit{$s$-constellation}, and denote it by the pair $(S_{\mf{c}},L_{\mf{c}})$ (observe that this matches the definition given in Section~\ref{sec:intro}). For a graph $G$, by an \textit{$(s,l)$-constellation in $G$} we mean an induced subgraph of $G$ which is an $(s,l)$-constellation.

By a \textit{$\mf{c}$-route} we mean a path $R$ in $\mf{c}$ with ends in $S_{\mf{c}}$ and with $R^*\subseteq V(\mca{L}_{\mf{c}})$, or equivalently, with $R^*\subseteq V(L)$ for some $L\in \mca{L}_{\mf{c}}$. For $d\in \poi$, we say that $\mf{c}$ is \textit{$d$-ample} if there is no $\mf{c}$-route of length at most $d+1$. We also say that $\mf{c}$ is \textit{ample} if $\mf{c}$ is $1$-ample. It follows that $\mf{c}$ is $1$-ample if and only if no two vertices in $S_{\mf{c}}$ have a common neighbor in $V(\mca{L}_{\mf{c}})$ (again, note that this is consistent with the definition given in Section~\ref{sec:intro}).

We say that $\mf{c}$ is \textit{interrupted} if there is a linear order $\pre$ on $S_{\mf{c}}$ such that for all $x,y,z\in S_{\mf{c}}$ with $x\prec y\prec z$ and every $\mf{c}$-route $R$ from $x$ to $y$, the vertex $z$ has a neighbor in $R$ (see Figure~\ref{fig:interrupted}). For $q\in \poi$, we say that a $\mf{c}$ is \emph{$q$-zigzagged} if there is a linear order $\pre$ on $S_{\mf{c}}$ such that for all $x,y\in S_{\mf{c}}$ with $x\prec y$ and every $\mf{c}$-route $R$ from $x$ to $y$, there are fewer than $q$ vertices $z\in S_{\mf{c}}$ where $x\prec z\prec y$ and $z$ has no neighbor in $R$.

The main result of \cite{tw16} is the following:

\begin{theorem}[Chudnovsky, Hajebi, Spirkl \cite{tw16}]\label{thm:motherKtt}
For all $l,l',r,s,s'\in \poi$, there are constants $f_{\ref{thm:motherKtt}}=f_{\ref{thm:motherKtt}}(l,l',r,s,s')\in \poi$ and $g_{\ref{thm:motherKtt}}=g_{\ref{thm:motherKtt}}(l,l',r,s,s')\in \poi$ with the following property. Let $G$ be a graph and assume that there is an induced $(f_{\ref{thm:motherKtt}}, g_{\ref{thm:motherKtt}})$-model in $G$. Then one of the following holds.
   \begin{enumerate}[\rm (a)]
        \item\label{thm:motherKtt_a}There is an induced subgraph of $G$ isomorphic to either $K_{r,r}$, a subdivision of $W_{r\times r}$, or the line graph of a subdivision of $W_{r\times r}$.
        \item\label{thm:motherKtt_b} There is an ample interrupted $(s,l)$-constellation $G$. 
    \item\label{thm:motherKtt_c} There is an ample $2r^2$-zigzagged $(s',l')$-constellation in $G$.
    \end{enumerate}
\end{theorem}

We will also need the following lemma from \cite{tw16}. We include the proof as it is short:

\begin{lemma}[Chudnovsky, Hajebi, Spirkl \cite{tw16}]\label{lem:eez_zigzag2}
Let $c,q\in \poi$ and let  $\mf{c}$ be an ample $q$-zigzagged $\left(2c+3q, 2c{c+q-1\choose c}\right)$-constellation. Then there are anticomplete subsets $X,Y$ of $V(\mf{c})$ with $\tw(X),\tw(Y)\geq c$.
\end{lemma}
\begin{proof}
    Since $|S_{\mf{c}}|=2c+3q$, we may choose $x_1,x_2\in S_{\mf{c}}$ and pairwise disjoint subsets $Q,S_1,S_2\subseteq S_{\mf{c}}$ such that $|Q|=q$, $|S_1|=|S_2|=c+q-1$ and $x_1\prec S_1\prec Q\prec S_2\prec x_2$. 
    
    For each $i\in \{1,2\}$ and every $L\in \mca{L}_{\mf{c}}$, let $R_{i,L}$ be a $\mf{c}$-route from $x_i$ to a vertex in $Q$ with $|R_{i,L}|$ as small as possible.  We claim that:

\sta{\label{st:eachpathanti} For every $L\in \mca{L}_{\mf{c}}$, the sets $S_1$ and $R^*_{2,L}$ are anticomplete in $\mf{c}$, and the sets $S_2$ and $R^*_{1,L}$ are anticomplete in $\mf{c}$.}

Suppose not. Then, by symmetry, we may assume that some for some $L\in \mca{L}_{\mf{c}}$, there is a vertex $u\in S_2$ with a neighbor in $R^*_{1,L}$. Let $y\in Q$ be the end of $R_{1,L}$ other than $x$. Since $\mf{c}$ is ample, it follows that there is a $\mf{c}$-route $R'$ from $x_1$ to $u$ with $R'^*\subseteq R_{1,L}^*\setminus N_R(y)$. Since $\mf{c}$ is $q$-zigzagged, and since $x_1\prec Q\prec u$, it follows that some vertex $z\in Q$ has a neighbor in $R'^*$. Consequently, there is a $\mf{c}$-route $R''$ from $x$ to $z\in Q_1$ with $R''^*\subseteq R'^*\subseteq R_{1,L}^*\setminus N_R(y)$, and so $|R''|<|R_{1,L}|$. This violates the choice of $R_{1,L}$, hence proving \eqref{st:eachpathanti}.
\medskip

Let $\mca{L}_1,\mca{L}_2\subseteq \mca{L}_{\mf{c}}$ be disjoint with $|\mca{L}_1|=|\mca{L}_2|=c{c+q-1\choose c}$.  Since $\mf{c}$ is $q$-zigzagged and since $|S_1|=|S_2|=c+q-1$, it follows that for every $i\in \{1,2\}$ and every $L\in \mca{L}_{i}$, there is a $c$-subset $S'_{i,L}$ of $S_i$ such that every vertex in $S'_{i,L}$ has a neighbor in $R^*_{i,L}$. Since $|\mca{L}_1|=|\mca{L}_2|=c{c+q-1\choose c}$, it follows that for every $i\in \{1,2\}$, there is a $c$-subset $S'_i$ of $S_i$ and a $c$-subset $\mca{L}'_i$ of $\mca{L}_i$ such that for every $L\in \mca{L}'_i$, we have $S'_{i,L}=S'_i$.

Now, let
$$X=(S'_1,\{R^*_{1,L}:L\in \mca{L}'_1\})$$
and let
$$Y=(S'_2,\{R^*_{2,L}:L\in \mca{L}'_2\}).$$

By \eqref{st:eachpathanti} and since $\mf{c}$ is a constellation, it follows that $X$ and $Y$ are anticomplete $(c,c)$-constellations in $\mf{c}$. In particular, $X$ and $Y$ are anticomplete induced subgraphs of $\mf{c}$ each with an induced minor isomorphic to $K_{c,c}$. Hence, we have $\tw(X),\tw(Y)\geq c$. This completes the proof of Lemma~\ref{lem:eez_zigzag2}.
\end{proof}

We are now ready to prove Theorem~\ref{thm:EEZ_isg}, which we restate:

\eezisg*
\begin{proof}
    Let $r=\max\{2c,t\}$. Let $s'=2c+6r^2$ and let
$l'=2c\binom{c+2r^2-1}{c}$. Let
$$b_1=f_{\ref{thm:motherKtt}}(1,l',r,t,s'),$$
let
$$b_2=g_{\ref{thm:motherKtt}}(1,l',r,t,s')$$
and let $b=\max\{b_1,b_2\}$. We claim that
$$f_{\ref{thm:EEZ_isg}}=f_{\ref{thm:EEZ_isg}}(c,t)=f_{\ref{thm:EEZ_candidate}}(t+1,b,c)$$
satisfies the theorem.
\medskip

Let $G$ be a graph with $\tw(G)>f_{\ref{thm:EEZ_isg}}$. By Theorem~\ref{thm:EEZ_candidate}, $G$ is not a $(t+1,b,c)$-candidate. If $G$ has an induced subgraph isomorphic to $K_{t+1}$, then \ref{thm:EEZ_isg}\ref{thm:EEZ_isg_a} holds. Also, if there are anticomplete induced subgraphs $X,Y$ of $G$ with $\tw(X),\tw(Y)\geq c$, then \ref{thm:EEZ_isg}\ref{thm:EEZ_isg_c} holds. Therefore, we may assume that there is an induced $(b,b)$-model in $G$. In particular, there is an induced $(b_1,b_2)$-model in $G$. From this, Theorem~\ref{thm:motherKtt} and the choice of $b_1, b_2$ and $r$, we deduce that one of the following holds.
\begin{itemize}
   \item There is an induced subgraph of $G$ isomorphic to either $K_{t,t}$, a subdivision of $W_{2c\times 2c}$, or the line graph of a subdivision of $W_{2c\times 2c}$.
        \item There is an ample interrupted $t$-constellation $G$. 
    \item There is an ample $2r^2$-zigzagged $(s',l')$-constellation in $G$.
\end{itemize}

Assume that the first bullet holds. Then either $G$ has an induced subgraph isomorphic to $K_{t,t}$, or there are anticomplete induced subgraphs $X,Y$ of $G$ such that both $X$ and $Y$ are isomorphic to either a subdivision of $W_{c\times c}$ or the line graph of a subdivision of $W_{c\times c}$. It follows that either \ref{thm:EEZ_isg}\ref{thm:EEZ_isg_a} or \ref{thm:EEZ_isg}\ref{thm:EEZ_isg_c} holds.

Since the second bullet and \ref{thm:EEZ_isg}\ref{thm:EEZ_isg_b} are identical, we may assume that the third bullet holds; that is, there is an ample and $2r^2$-zigzagged $(s',l')$-constellation $\mf{c}$ in $G$. But then by Lemma~\ref{lem:eez_zigzag2} and the choice of $s'$ and $l'$, there are anticomplete subsets $X,Y$ of $V(G)$ with $\tw(X),\tw(Y)\geq c$, and so \ref{thm:EEZ_isg}\ref{thm:EEZ_isg_c} holds. This completes the proof of Theorem~\ref{thm:EEZ_isg}.
\end{proof}

\section{Acknowledgements}
 This work was partly done during the 2024 Barbados Graph Theory Workshop at the Bellairs Research Institute of McGill University, in Holetown, Barbados. We thank the organizers for inviting us and for creating a stimulating work environment.

\bibliographystyle{plain}
\bibliography{ref}

\begin{thebibliography}{10}

\bibitem{tw7}
Tara Abrishami, Bogdan Alecu, Maria Chudnovsky, Sepehr Hajebi, and Sophie Spirkl.
\newblock Induced subgraphs and tree decompositions {VII}. {B}asic obstructions in {$H$}-free graphs.
\newblock {\em J. Combin. Theory Ser. B}, 164:443--472, 2024.

\bibitem{tw11}
Bogdan Alecu, Maria Chudnovsky, Sepehr Hajebi, and Sophie Spirkl.
\newblock {I}nduced subgraphs and tree decompositions {XI}. {L}ocal structure in even-hole-free graphs of large treewidth.
\newblock {\em {\rm Manuscript available at \url{https://arxiv.org/abs/2309.04390}}}, 2023.

\bibitem{tw9}
Bogdan Alecu, Maria Chudnovsky, Sepehr Hajebi, and Sophie Spirkl.
\newblock Induced subgraphs and tree decompositions {IX}. {G}rid theorem for perforated graphs.
\newblock {\em Adv. Comb.}, pages Paper No. 3, 40, 2025.

\bibitem{tw12}
Bogdan Alecu, Maria Chudnovsky, Sepehr Hajebi, and Sophie Spirkl.
\newblock Induced subgraphs and tree decompositions {XII}. {G}rid theorem for pinched graphs.
\newblock {\em Innov. Graph Theory}, 2:1--23, 2025.

\bibitem{berge}
Claude Berge.
\newblock Some classes of perfect graphs.
\newblock In {\em Combinatorial {M}athematics and its {A}pplications ({P}roc. {C}onf., {U}niv. {N}orth {C}arolina, {C}hapel {H}ill, {N}.{C}., 1967)}, volume No. 4 of {\em University of North Carolina Monograph Series in Probability and Statistics}, pages 539--552. Univ. North Carolina Press, Chapel Hill, NC, 1969.

\bibitem{deathstar}
Marthe Bonamy, \'{E}douard Bonnet, Hugues D\'{e}pr\'{e}s, Louis Esperet, Colin Geniet, Claire Hilaire, St\'{e}phan Thomass\'{e}, and Alexandra Wesolek.
\newblock Sparse graphs with bounded induced cycle packing number have logarithmic treewidth.
\newblock {\em J. Combin. Theory Ser. B}, 167:215--249, 2024.

\bibitem{ti2}
Maria Chudnovsky, Sepehr Hajebi, Daniel Lokshtanov, and Sophie Spirkl.
\newblock Tree independence number {II}. {T}hree-path-configurations.
\newblock {\em J. Combin. Theory Ser. B}, 176:74--96, 2026.

\bibitem{tw16}
Maria Chudnovsky, Sepehr Hajebi, and Sophie Spirkl.
\newblock {I}nduced subgraphs and tree decompositions {XVI}. {C}omplete bipartite induced minors.
\newblock {\em {\rm Manuscript available at \url{https://arxiv.org/abs/2410.16495}}}, 2024.

\bibitem{diestel}
Reinhard Diestel.
\newblock {\em Graph theory}, volume 173 of {\em Graduate Texts in Mathematics}.
\newblock Springer, Berlin, fifth edition, 2018.

\bibitem{eez}
M.~El-Zahar and P.~Erd\H{o}s.
\newblock On the existence of two nonneighboring subgraphs in a graph.
\newblock {\em Combinatorica}, 5(4):295--300, 1985.

\bibitem{ErdSze}
Paul Erd\H{o}s and George Szekeres.
\newblock A combinatorial problem in geometry.
\newblock {\em Compositio Math.}, 2:463--470, 1935.

\bibitem{productramsey}
Ronald~L. Graham, Bruce~L. Rothschild, and Joel~H. Spencer.
\newblock {\em Ramsey theory}.
\newblock Wiley Series in Discrete Mathematics and Optimization. John Wiley \& Sons, Inc., Hoboken, NJ, paperback edition, 2013.

\bibitem{pinned}
Sepehr Hajebi.
\newblock Induced subdivisions with pinned branch vertices.
\newblock {\em European J. Combin.}, 124:Paper No. 104072, 2025.

\bibitem{korhonen2023grid}
Tuukka Korhonen.
\newblock Grid induced minor theorem for graphs of small degree.
\newblock {\em Journal of Combinatorial Theory, Series B}, 160:206--214, 2023.

\bibitem{kuhn2004induced}
Daniela Kuhn and Deryk Osthus.
\newblock Induced subdivisions in {$K_{s,s}$}-free graphs of large average degree.
\newblock {\em Combinatorica}, 24(2):287--304, 2004.

\bibitem{multiramsey}
Frank~P. Ramsey.
\newblock On a {P}roblem of {F}ormal {L}ogic.
\newblock {\em Proc. London Math. Soc. (2)}, 30(4):264--286, 1929.

\bibitem{GMV}
Neil Robertson and Paul Seymour.
\newblock Graph minors. {{V}}. {{E}}xcluding a planar graph.
\newblock {\em Journal of Combinatorial Theory, Series B}, 41(1):92--114, 1986.

\end{thebibliography}

\end{document}